\documentclass[11pt]{amsart}

\usepackage{amsmath,amssymb,latexsym,soul,cite,mathrsfs}

\usepackage{color,enumitem,graphicx}
\usepackage[colorlinks=true,urlcolor=blue,
citecolor=red,linkcolor=blue,linktocpage,pdfpagelabels,
bookmarksnumbered,bookmarksopen]{hyperref}
\usepackage[english]{babel}

\usepackage[left=2.9cm,right=2.9cm,top=2.8cm,bottom=2.8cm]{geometry}
\usepackage[hyperpageref]{backref}

\usepackage[colorinlistoftodos]{todonotes}
\makeatletter
\providecommand\@dotsep{5}
\def\listtodoname{List of Todos}
\def\listoftodos{\@starttoc{tdo}\listtodoname}
\makeatother

\numberwithin{equation}{section}

\newtheorem{theorem}{Theorem}[section]
\newtheorem{proposition}[theorem]{Proposition}
\newtheorem{lemma}[theorem]{Lemma}
\newtheorem{corollary}[theorem]{Corollary}

\newtheorem{definition}{Definition}[section]

\begin{document}

\title[ Stability of solutions for a parabolic problem involving fractional $p$-Laplacian...]
{Stability of solutions for a parabolic problem involving fractional $p$-Laplacian with logarithmic nonlinearity}
\author{Tahir Boudjeriou}
\address[Tahir Boudjeriou]
{\newline\indent Department of Mathematics
	\newline\indent Faculty of Exact Sciences
	\newline\indent Lab. of Applied Mathematics
	\newline\indent University of Bejaia, Bejaia, 6000, Algeria 
	\newline\indent
	e-mail:re.tahar@yahoo.com}

\pretolerance10000


\begin{abstract}
	\noindent In this paper, we study the following Dirichlet problem for a parabolic equation involving fractional $p$-Laplacian with logarithmic nonlinearity 
	\begin{equation*}\label{eq}\left\{
	\begin{array}{llc}
	u_{t}+(-\Delta)^{s}_{p}u+|u|^{p-2}u=|u|^{p-2}u\log(|u|) & \text{in}\ & \Omega,\;t>0 , \\
	u =0 & \text{in} & \mathbb{R}^{N}\backslash \Omega,\;t > 0, \\
	u(x,0)=u_{0}(x), & \text{in} &\Omega , 
	\end{array}\right.
	\end{equation*}
where $\Omega \subset \mathbb{R}^N \, ( N\geq 1)$ is a bounded domain with Lipschitz boundary and $2\leq p< \infty$. The local existence will be done by using the Galerkin approximations. By combining the potential well theory with the Nehari manifold we establish the existence of global solutions. Then, by virtue of a differential inequality technique, we prove that the local solutions blow-up in finite time with arbitrary negative initial energy and suitable initial values. Moreover, we give decay estimates of global solutions. The main difficulty here is the lack of logarithmic Sobolev inequality concerning fractional $p$-Laplacian.
 \end{abstract}
\thanks{}
\maketitle	
\section{Introduction and the main results} 
This paper is concerned with the following parabolic equation involving fractional $p$-Laplacian with logarithmic nonlinearity 
\begin{equation}\label{eq1}\left\{
\begin{array}{llc}
u_{t}+(-\Delta)^{s}_{p}u+|u|^{p-2}u=|u|^{p-2}u\log(|u|) & \text{in}\ & \Omega,\;t>0 , \\
u =0 & \text{in} & \mathbb{R}^{N}\backslash \Omega,\;t > 0, \\
u(x,0)=u_{0}(x), & \text{in} &\Omega , 
\end{array}\right.
\end{equation}
where $\Omega \subset \mathbb{R}^N \, ( N\geq 1)$ is a bounded domain with Lipschitz boundary $\partial \Omega$, $u_{0}\neq 0$ is the initial function on $\Omega$, and $(-\Delta)_{p}^{s}$ is the fractional $p$-Laplacian which is nonlinear nonlocal operator defined on smooth functions by 
\begin{eqnarray}
(-\Delta)_{p}^{s}\varphi(x) &=& 2\lim_{\epsilon \downarrow 0} \int_{\mathbb{R}^{N}\backslash B_{\epsilon}(x)}\frac{|\varphi(x)-\varphi(y)|^{p-2}(\varphi(x)-\varphi(y))}{|x-y|^{N+sp}}\,dy.
\end{eqnarray}
This definition is consistent, up to a normalization constant depending on $N$ and $s$. We refer the reader to (\cite{DV}, \cite{Cl}) and the references therein for further details on the fractional Laplacian and on the fractional Sobolev spaces. Throughout the paper, without further mentioning, we always assume that $s\in (0,1)$, $N>sp$ and $2\leq p<\infty.$  

The interest in studying problems like $\eqref{eq1}$ relies not only on mathematical purposes but also on their significance in real models, as explained by \textit{Caffarelli} in \cite{L.C} and  \textit{G. Gilboa et al} in \cite{Gi}. \textit{Applebaum} in \cite{J.B} stated that the fractional $2-$Laplacian operator of the form $(-\Delta)^{s}$, $0<s<1$, is an infinitesimal generator of stable L\'evy processes.

In recent years, the global existence and blow-up of solutions to the following model 
$$
\left\{\begin{array}{l}
u_{t}-\Delta u=f(u), \;\; x\in \Omega, \; t>0, \\
u(x,t)=0, \;\;\;\; x\in \partial \Omega, \;t> 0, \\
u(x,0)=u_{0}(x), \;\;\; x\in \Omega
\end{array}\right.\leqno{(M_{1})}
$$
has been studied extensively by many authors, see for example  (\cite{LI}, \cite{Py}, \cite{Tan}, \cite{BL}, \cite{CT}) and the references therein where the authors have assumed the following conditions on the nonlinearity $f(u)$ :
\begin{enumerate}
	\item $f\in C^{1}(\mathbb{R})$ and $f(0)=f'(0)=0.$
	\item $f$ is monotone and convex for $u>0$, concave for $u<0$.
	\item $(p+1)F(s)< sf(s)$, and $|sf(s)|\leq \mu F(s)$, where $2<p<\mu< \frac{2N}{N-2}, \; F(s)=\int^{s}_{0}f(r)\,dr.$ 
\end{enumerate}
Recently, the logarithmic heat equation given by 
$$ v_{t}=\Delta v+|v|^{p-2}\log|v|, \;\; v :\mathbb{R}^{N}\times [0,+\infty)\rightarrow \mathbb{R},\; p,N\geq 2.$$
has also received special attention because it appears in a lot of physical applications, such as nuclear physics, transport and diffusion phenomena, theory of superfluidity (see \cite{ZHH} and the references therein). It is worth noting that when $f(u)$ is a logarithmic nonlinear function, i.e, $f(u)=|u|^{p-2}\log|u|$, the above assumptions  $(1)-(3)$ are not satisfied. So it is difficult to deal with this type of nonlinearity. For parabolic equations involving classical Laplacian with logarithmic nonlinearity, we refer to (\cite{Con}, \cite{Ch}, \cite{Chh2}) where the authors have used the Sobolev logarithmic inequality \cite{Del} to study the global existence and blow-up of solutions, as well as they, compared the difference between logarithmic nonlinearity and polynomial nonlinearity. 

 It is worth mentioning that by using variational methods many authors have studied the stationary case of problem $\eqref{eq1}$ in unbounded domains where $p\geq 2$ and $0<s \leq 1$, in that direction we would like to mention  ( \cite{W.Ch},\cite{T}, \cite{T2}, \cite{CO}, \cite{AD}, \cite{Squa})  for the interested reader.
 
We point out that in the last years many authors have obtained important results on the fractional $p-$Laplacian in bounded or unbounded domains with polynomial type nonlinearities, for example see  (\cite{Jac}, \cite{BB}, \cite{Wr}, \cite{Wrr},\cite{Jun}) and the references therein. With the help of potential well theory, \textit{Fu} and \textit{Pucci} \cite{FU}, studied the existence of global weak solutions and established the vacuum isolating and blow-up of strong solutions for the following class of problem 
$$
\left\{\begin{array}{l}
u_{t}+ (-\Delta)^{s} u=|u|^{p-2}u, \;\; x\in \Omega, \; t>0, \\
u(x,t)=0, \;\;\;\; x\in \mathbb{R}^{N}\backslash \Omega, \;t> 0, \\
u(x,0)=u_{0}(x), \;\;\; x\in \Omega
\end{array}\right.\leqno{(M_{2})}
$$
for $s\in (0,1)$, $N>2s$ and $2< p\leq 2_{s}^{*}=2N/(N-2s).$  In (\cite{Hng}, \cite{MII}, \cite{BZ}) by using the Galerkin method combined with the potential well theory, the authors have studied  the existence of global weak solutions for the degenerate Kirchhoff-type diffusion problems involving fractional Laplacian. Moreover, they obtained also estimates for the lower and upper bounds of the blow-up time. In \cite{JMM}, by using the sub-differential approach, \textit{Maz\'on et al} obtained the well-posedness of solutions for problem $\eqref{eq1}$ with $f(x,t)$ instead of $|u|^{p-2}\log(|u|)$. Moreover, the large-time behavior of solutions also are considered.\par  

Motivated by the above works, in the pressent manscript by using the potential well theory combined with the Nehari manifold we dicuss the global existence and finite time blow-up for the solutions of problem $\eqref{eq1}$. To the best of our knowledge, this is the first result in the literature to investigate the global existence and blow-up of solutions in the study of diffusion fractional $p-$Laplacian with logarithmic nonlinearity.

Hereafter, we denote by $X_{0}=W_{0}^{s,p}(\Omega)\backslash \{0\}$
which will be introduced in section $2$. The energy functional $E :X_{0}\rightarrow \mathbb{R}$ associated with problem $\eqref{eq1}$ is given by 
\begin{equation}\label{d2}
E(u(t))=\frac{1}{p}\int_{\mathbb{R}^{N}}\int_{\mathbb{R}^{N}}\frac{|u(x,t)-u(y,t)|^{p}}{|x-y|^{N+sp}}\,dydx+\frac{1}{p}\int_{\Omega} |u(t)|^{p}\,dx-\frac {1}{p}\int_{\Omega} |u(t)|^{p}\log(|u(t)|)\,dx+\frac{1}{p^{2}}\int_{\Omega} |u(t)|^{p}\,dx,
\end{equation}
We define the Nehari functional $I : X_{0}\rightarrow \mathbb{R}$ by 
\begin{equation}\label{d3}
I(u(t))=\int_{\mathbb{R}^{N}}\int_{\mathbb{R}^{N}}\frac{|u(x,t)-u(y,t)|^{p}}{|x-y|^{N+sp}}\,dydx+\int_{\Omega} |u(t)|^{p}\,dx-\int_{\Omega} |u(t)|^{p}\log(|u(t)|)\,dx.
\end{equation}
The potential well associated with problem $\eqref{eq1}$ is the set
\begin{equation}\label{pt7}
\mathcal{W} =\{u\in X_{0}\,:\, E(u)< d,\; I(u)>0 \}.
\end{equation}
where $d$ is the depth of the potential well. The exterior of the potential well is the set 
\begin{equation}\label{ZX}
Z =\{u\in X_{0}\,:\, E(u)< d,\; I(u)<0 \}.
\end{equation}
Related to the functional $E$, we have the well-known {\it Nehari manifold.}
$$\mathcal{N}=\{u\in X_{0}:\; I(u)=0\}.$$ We define 
\begin{equation}\label{V1}
d=\inf_{u\in \mathcal{N}} E(u).
\end{equation}
It is worthwhile to mention that the potential well method was introduced in (\cite{DHH}) to obtain the global existence for nonlinear hyperbolic equations. The most important and typical work on the potential well is due to Payne and Sattinger in \cite {Py} where the authors have studied the initial boundary value problem of semilinear hyperbolic equations and semilinear parabolic equations.
We refer the reader to \cite{I}, where the potential well method was extended to obtain global existence and nonexistence results for the parabolic equations. To state the main results, we need the following two definitions.
\begin{definition}(Weak solution )\label{de1} $u=u(x,t)$ is called a weak solution of problem $\eqref{eq1}$ in $\Omega \times (0,T_{*})$, if 
	$u\in L^{\infty}(0,T_{*}, W_{0}^{s, p}(\Omega))$ with $u_{t}\in L^{2}(0,T_{*},L^{2}(\Omega))$ and satisfies the problem $\eqref{eq1}$ in the distribution sense, i.e 
	\begin{equation}\label{DEF}
	\int_{\Omega}u_{t}v\,dx+K^{s,p}(u,v)+\int_{\Omega} |u|^{p-2}uv\,dx=\int_{\Omega} |u|^{p-2}u\log(|u|)v\,dx,\;\; \forall v\in W_{0}^{s,p}(\Omega), \; \text{a.e.}\, t\in (0,T_{*}),
	\end{equation}
	where $u(0,x)=u_{0}(x)\in X_{0},$ where $K^{s, p}(u,v)$ will be introduced in section $2$.
\end{definition}
\begin{definition}(Maximal existence time)\label{de2} Let $u(t)$ be a weak solution of problem $\eqref{eq1}$. We define the maximal existence time $T_{\max}$ of $u$ as follows : 
	\begin{enumerate}
		\item If $u$ exists for all $ 0\leq t< +\infty$, then $T_{\max}=+\infty$; 
		\item If there exits a $t_{0}\in (0, +\infty)$ such that $u$ exits for $0\leq t< t_{0}$, but does not exist at $t=t_{0}$, then $T_{\max}=t_{0}$.
	\end{enumerate}
\end{definition}
Based on the above preparations, the main results of this paper are the following theorems.
\begin{theorem}\label{th1}
	Let $u_{0}\in X_{0}$. Then there exists a postive constant $T_{*}$ such that the problem $\eqref{eq1}$ has weak solution $u(x,t)$ on $\Omega \times(0,T_{*})$ in the sense of definition \ref{de1}. Furthermore, $u(x,t)$ satisfies the energy inequality 
	\begin{equation}\label{en1}
	\int_{0}^{t}\|u_{s}(s)\|_{2}^{2}\,ds+E(u(t))\leq E(u_{0}),\;\;\text{a.e.}\, t\in [0,T_{*}].
	\end{equation}
\end{theorem}
\begin{theorem}\label{th3}
	Let $u_{0}\in \mathcal{W}$. Then the problem $\eqref{eq1}$ admits a global weak solution and satisfying the energy estimate 
	\begin{equation}\label{eng}
	\int_{0}^{t}\|u_{s}(s)\|_{2}^{2}\,ds+E(u(t))\leq E(u_{0}), \; a.e.\; t\geq 0.
	\end{equation}
	Moreover, the solution decays polynomially, namely.
	If $E(u_{0})< M$, then we have 
	\begin{equation}
	\|u(t)\|_{2}\leq \|u_{0}\|_{2}\left( \frac{p}{2(1+\kappa(p-2)\|u_{0}\|^{p-2}_{2}t)}\right)^{\frac{1}{p-2}},\;\;t\geq 0, 
	\end{equation}
	where $\kappa =|\Omega|^{2-p}\left(1-C\left(\frac{1}{2}\right)\left(p^{2}E(u_{0})\right)^{\gamma-1}\right)> 0.$
\end{theorem}
\begin{theorem}\label{th4}
	Let $u_{0}\in Z$. Assume that $u(x,t)$ be a local weak solution of the problem $\eqref{eq1}$ corresponding to this initial function and $u$ satisfies the energy inequality 
	\begin{equation}\label{ENR}
	\int_{0}^{t}\|u_{s}(s)\|^{2}_{2}\,ds+E(u(t))\leq E(u_{0}), \;\;\;\forall t\in[0,T_{\max}).
	\end{equation}
	If $E(u_{0})\leq 0$, then $u(x,t)$ is not continued globally, i.e.
	$$ \lim_{t\rightarrow T^{-}_{\max}}\|u(t)\|^{2}_{2}=+\infty.$$
\end{theorem}
The difficulty here is the lack of logarithmic Sobolev inequality which seems there is no logarithmic Sobolev inequality concerning fractional $p$- Laplacian yet.\par 
It is worthwhile to remark that the proofs of the main results do not exploit any monotonicity assumption but rely on a compactness argument in combination with the regularity of the Galerkin shame and the nonlocal character of the operator. This argument has used for the first time by \textit{E. Emmrich} and \textit{D. Puhst} in (\cite{E},\cite{Pus}).\par 
The rest of the paper is organized as follows. In Section $2$, we give some preliminary results that will be used throughout the paper. In section $3$, we obtain the local existence of weak solutions of problem $\eqref{eq1}$ by the Galerkin method. In Section $4$, under suitable conditions, we show that the weak solutions of problem $\eqref{eq1}$ exist globally. Moreover, we give a decay estimate of global solutions. In, section $5$, we prove that the weak solutions of problem $\eqref{eq1}$ blow-up in finite time under some appropriate conditions.
\section{Preliminaries}
\subsection{Fractional Sobolev spaces and fractional $p$-Laplacian}
In this subsection, we recall some necessary properties of fractional Sobolev spaces and the fractional $p$-Laplacian which will be used later, see (\cite{DV}, \cite{Cl}) for more details.

Let $\Omega$ be an open set in $\mathbb{R}^{N}$ with Lipschitz boundary boundary. For any $p\in (1,\infty)$ and any $0< s< 1$, we consider the fractional Sobolev space 

$$ W^{s,p}(\Omega)=\left\{u\in L^{p}(\Omega)\;\;u\text{measurable and}\,[u]_{s,p}< \infty \right\}, $$
where the $(s,p)$-Gagliardo seminorm is defined as 
$$ [u]_{s,p}=\left(\int_{\mathbb{R}^{N}} \int_{\mathbb{R}^{N}} \frac{|u(y)-u(x)|^{p}}{|x-y|^{N+sp}}\,dxdy\right)^{1/p},$$
In oder to obtain the existence of weak solutions for $\eqref{eq1}$, we consider the subspace of $W^{s,p}(\Omega)$ 
 $$ W_{0}^{s,p}(\Omega)=\left\{u\in L^{p}(\Omega),\;[u]_{s,p}< \infty,\;u=0\;\text{a.e. in}\;\mathbb{R}^{N}\backslash \Omega \right\}.$$
Equipped with the norm 
$$\|u\|:=[u]_{s,p}.$$
$W^{s,p}_{0}(\Omega)$ is a separable reflexive Banach space. Since $\Omega$ has a Lipschitz boundary we have   
$$ W^{s,p}_{0}(\Omega)=\overline{C^{\infty}_{0}(\Omega)}^{W^{s,p}(\Omega)}.$$
The functions in the space $ W_{0}^{s,p}(\Omega)$ can be defined in the whole space $W^{s,p}_{0}(\mathbb{R}^{N})$ by extending then by zero outside of $\Omega$.
We will write, as usual $p_{s}^{*}=\frac{Np}{N-sp}$, to denote the fractional critical exponent for $1\leq p< \frac{N}{s}$. If $1\leq r\leq p_{s}^{*}$ then we have the continuous immersion $W_{0}^{s,p}(\Omega)\hookrightarrow L^{r}(\Omega)$, that is compact for $1\leq r< p_{s}^{*}$.\\
If we assume that the integral in the definition of $(-\Delta)^{s}_{p}u$ exists, then for any $ \varphi \in W^{s,p}(\mathbb{R}^{N})$ due to the symmetry of the kernel we have the following integration by parts formula 
\begin{eqnarray*}
	\int_{\mathbb{R}^{N}} (-\Delta)_{p}^{s}u(x)\varphi(x)\,dx&=&2\int_{\mathbb{R}^{N}} \int_{\mathbb{R}^{N}} \frac{|u(x)-u(y)|^{p-2}(u(x)-u(y))}{|x-y|^{N+sp}}\,dy\varphi(x)dx\\
	&=& \int_{\mathbb{R}^{N}} \int_{\mathbb{R}^{N}} \frac{|u(y)-u(x)|^{p-2}(u(y)-u(x))}{|x-y|^{N+sp}}(\varphi(x)-\varphi(y))\,dydx.
\end{eqnarray*}
 \begin{proposition}(\cite{Pus})\label{prp1}
	The nonlinear form $K^{s,p} : W_{0}^{s,p}(\Omega) \times W_{0}^{s,p}(\Omega) \rightarrow \mathbb{R}$ given by 
	\begin{equation}\label{e1}
	K^{s,p}(u,v)=\int_{\mathbb{R}^{N}} \int_{\mathbb{R}^{N}} \frac{|u(x)-u(y)|^{p-2}(u(x)-u(y))}{|x-y|^{N+sp}}(v(x)-v(y))\,dydx,
	\end{equation}
is well-defined, bounded, continuous in its first and second argument, and monotone there holds for all $u,v\in W_{0}^{s,p}(\Omega)$ 
	\begin{equation}\label{eqq1}
	|K^{s,p}(u,v)|\leq [u]^{p-1}_{s,p}[v]_{s,p}, 
	\end{equation}
	\begin{equation}\label{eqq2}
	K^{s,p}(u,u-v)-K^{s,p}(v,u-v)\geq0.
	\end{equation}
	Moreover, there holds for $u\in W^{s,p}_{0}(\Omega)$
	\begin{equation}
	K^{s,p}(u,u)=[u]_{s,p}^{p}.
	\end{equation}
	The potential $\Phi^{s,p}: W_{0}^{s,p}(\Omega)\rightarrow \mathbb{R}$ given by 
	$$ \Phi^{s,p}(u)=\frac{1}{p}[u]_{s,p}^{p}$$
	is well-defined, bounded, nonnegative, and has the G\^ateaux derivative
	$$ \langle (\Phi^{s,p})'(u),v\rangle=K^{s,p}(u,v), \;\; u,v\in W_{0}^{s,p}(\Omega).$$
\end{proposition}
\begin{lemma}(\cite{Pus})\label{lem1}
	For any $0< \eta < \min\left( \frac{1-s}{p-1}, s\right)$ the form $K^{s,p}$ given by $\eqref{e1}$ is also well-defined on $W^{s-\eta, p}_{0} (\Omega)\times W^{s+\eta(p-1), p}_{0} (\Omega)$, bounded and continuous in both of its arguments.
\end{lemma}
The next lemma is taken from \cite[Theorem 3, p.303]{Eva}.
\begin{lemma}\label{le}
	Let $H$ be a Hilbert space. Suppose that $u\in L^{2}(0, T, H)$ and $u_{t}\in L^{2}(0, T, H)$. 
	Then the mapping $t\mapsto \|u(t)\|_{H}$ is absolutely continuous with 
	$$\frac{1}{2}\frac{d}{dt}\|u(t)\|_{H}^{2}= ( u_{t}(t), u(t) ) $$
\end{lemma}
We recall the following classical interpolation inequality from [\cite{LT}, Lemma 8.2].
\begin{lemma}\label{lemin}
	If $1\leq p_{0}< p_{\theta}< p_{1}\leq \infty$, then 
	\begin{equation}\label{Intr}
	\|u\|_{p_{\theta}}\leq \|u\|_{p_{0}}^{1-\theta}\|u\|_{p_{1}}^{\theta},
	\end{equation}
	for all $u\in L^{p_{0}}(\Omega)\cap L^{p_{1}}(\Omega)$ with $\theta\in (0,1)$ defined by $\frac{1}{p_{\theta}}=\frac{1-\theta}{p_{0}}+\frac{\theta}{p_{1}}.$
\end{lemma}
The proof of the following Lemma is straightforward.
\begin{lemma}\label{le1}
	Let $\varrho $ be a positive number. Then the following inequality holds 
	$$ |log(s)| \leq \frac{1}{\varrho}s^{\varrho}, $$
	for all $s\in [1,+\infty)$.
\end{lemma}
Throughout the paper, the letters $c$, $c_{i}$, $C$, $C_{i}$, $i=1, 2, \ldots, $ denote positive constants which vary from line to line, but are independent of terms that take part in any limit process. Furthermore, we will use these notations  
 $$ \|u\|_{p}=\|u\|_{L^{p}(\Omega)},\quad X_{0}=W_{0}^{s,p}(\Omega)\backslash \{0\}.$$
\subsection{Potential well in a variational stationary setting}
For simplicity, in this subsection, we consider the problem  $\eqref{eq1}$ in the stationary case. We point out that if we replace $u$ in this subsection by $u(t)$ for any $t\in [0, T)$, all the facts are still valid.\\
From lemmas \ref{lemin} and \ref{le1}, one  can check easily that the functionals $I$ and $E$ which have introduced in section $1$ are continuous on  $X_{0}$. Notice that 
\begin{equation}\label{u1}
E(u)=\frac{1}{p}I(u)+\frac{1}{p^{2}}\|u\|^{p}_{p}.
\end{equation}
Let $u\in X_{0}$ and let us consider the real function $j :\lambda \mapsto E(\lambda u)$ for $\lambda > 0$, defined by 
\begin{equation}\label{d4}
j(\lambda)=E(\lambda u)=\frac{\lambda^{p}}{p}[u]^{p}_{s,p}+\frac{\lambda^{p}}{p}\int_{\Omega} |u|^{p}\,dx-\frac {\lambda^{p}}{p}\int_{\Omega} |u|^{p}\log(|u|)\,dx-\frac{\lambda^{p}}{p}\log(\lambda)\int_{\Omega} |u|^{p}\,dx+\frac{\lambda^{p}}{p^{2}}\int_{\Omega} |u|^{p}\,dx
\end{equation}
Such maps are known as fibering maps which were introduced by Drabek and Pohozaev \cite{Dr}. The following lemma gives some properties of the real function $j(\lambda)$.
\begin{lemma}\label{le2}
	Let $u\in X_{0}$. Then we have 
	\begin{description}
		\item[1]$\lim_{\lambda\rightarrow 0^{+}}j(\lambda)=0$ and $\lim_{\lambda \rightarrow+\infty}j(\lambda)=-\infty$ 
		\item [2] there is a unique $\lambda^{*}=\lambda^{*}(u)> 0$ such that $j'(\lambda^{*})=0$.
		\item [3] $j(\lambda)$ is increasing on $(0,\lambda^{*})$, decreasing on $(\lambda^{*}, +\infty)$ and attains the maximum at $\lambda^{*}$.
		\item[4] $I(\lambda u)> 0$ for $0< \lambda < \lambda^{*}$, $I(\lambda u)< 0$ for $\lambda^{*}< \lambda < +\infty$ and $I(\lambda^{*}u)=0$.
	\end{description} 
\end{lemma}
\begin{proof}
	For $u\in X_{0}$, by definition of $j(\lambda)=E(\lambda u)$, it is clear that the first statement holds due to $\|u\|_{p}\neq 0$. Now, by differentiating $j(\lambda)$ we obtain 
	\begin{equation}
	\frac{d}{d\lambda}j(\lambda)=\lambda^{p-1}\left([u]_{s,p}^{p}+\|u\|^{p}_{p}-\int_{\Omega}|u|^{p}\log(|u|)\,dx-\log(\lambda)\|u\|^{p}_{p}\right).
	\end{equation}
	Therefore, by taking 
	$$ \lambda^{*}=\lambda^{*}(u)=\exp\left(\frac{[u]_{s,p}^{p}+\|u\|^{p}_{p}-\int_{\Omega}|u|^{p}\log(|u|)\,dx}{\|u\|^{p}_{p}}\right),$$
	the second and third statements can be shown easily. In order to show the fourth statement,  one can check that 
	$$ I(\lambda u)=\lambda j'(\lambda).$$
	The proof is now complete.
\end{proof}
The following lemma gives some properties of the \textit{Nehari} functional $I$.
\begin{lemma}\label{le3}
	Let $u\in X_{0}$. The following statements hold :
	\begin{description}
		\item[1] $0< \|u\|_{p}	< R$ then $I(u)> 0$.
		\item [2] If $I(u)< 0$ then $\|u\|_{p}> R$.
		\item [3] If $I(u)=0$ then $\|u\|_{p}\geq R.$
	\end{description}
	where $R$ is constant given by 
	$$ R:=\left(\frac{1}{C(\frac{1}{2})}\right)^{\frac{1}{p(\gamma-1)}}.$$
\end{lemma}
\begin{proof}
	First, we observe that 
	\begin{eqnarray}\label{rrr1}
	\nonumber\int_{\Omega} |u(t)|^{p}\log(|u(t)|)\,dx&=&\int_{\Omega_{1}}|u(t)|^{p}\log(|u(t)|)\,dx+\int_{\Omega_{2}}|u(t)|^{p}\log(|u(t)|)\,dx\\\nonumber
	&\leq &\frac{1}{\varrho}\int_{\Omega_{2}}|u(t)|^{p+\varrho}\,dx\\
	&\leq & \|u(t)\|^{p+\varrho}_{p+\varrho},
	\end{eqnarray}
	where $\varrho$ is chosen sufficiently small such that $0< \varrho < \frac{sp^{2}}{N}$ and $\Omega_{1}:=\{x\in \Omega\,: |u(x,t)|\leq 1\}$ and $\Omega_{2}:=\{x\in \Omega\,: |u(x,t)|> 1\}$. The choice of $\varrho$ ensures that $p< p+\varrho< p^{*}_{s}$. Thus, Lemma \ref{lemin} combined with the continuous embedding  $ W_{0}^{s,p}(\Omega)\hookrightarrow  L^{p_{s}^{*}}(\Omega),$ yields 
	\begin{eqnarray}\label{III1}
	\nonumber\int_{\Omega} |u(t)|^{p}\log(|u(t)|)\,dx &\leq& C\|u(t)\|_{p_{s}^{*}}^{\theta(p+\varrho)}\|u(t)\|_{p}^{(1-\theta)(p+\varrho)}\\
	&\leq &C_{1} [u(t)]_{p,s}^{\theta(p+\varrho)}\|u(t)\|_{p}^{(1-\theta)(p+\varrho)}
	\end{eqnarray}
	where $\theta =\frac{s\varrho}{Np(p+\varrho)}\in (0,1)$. Since $0< \varrho< \frac{sp^{2}}{N}$, it follows that $\theta(p+\varrho)< p$. By using Young inequality for any $\varepsilon \in (0,1)$ we obtain 
	\begin{equation}\label{eg11}
	\int_{\Omega} |u|^{p}\log(|u|)\,dx\leq \varepsilon[u]^{p}_{s,p}+C(\varepsilon)\left(\|u\|^{p}_{p}\right)^{\gamma},
	\end{equation}
	where 
	\begin{equation}\label{cr}
	\gamma :=\frac{(1-\theta)(p+\varrho)}{p-\theta(p+\varrho)}> 1.
	\end{equation}
	Combining $\eqref{eg11}$ with the definition of $I$, yields
	\begin{eqnarray}
	\nonumber I(u)&=&[u]^{p}_{s,p}+\|u\|^{p}_{p}-\int_{\Omega} |u|^{p}\log (|u|)\,dx\\
	&\geq & (1-\varepsilon)[u]^{p}_{s,p}+\|u\|_{p}^{p}\left( 1-C(\varepsilon)\left(\|u\|_{p}^{p}\right)^{\gamma-1}\right), .
	\end{eqnarray}
	Taking $\varepsilon =\frac{1}{2}$ we get 
	\begin{equation}\label{a1}
	I(u)\geq \|u\|_{p}^{p}\left( 1-C\left(\frac{1}{2}\right)\left(\|u\|_{p}^{p}\right)^{\gamma-1}\right).
	\end{equation}
	If $0< \|u\|_{p}< R$, then it turns out from $\eqref{a1}$ that 
	$$ I(u)> 0.$$
	Now, assuming that $I(u)< 0$, again from   $\eqref{a1}$ it follows that 
	$$ 1-C\left(\frac{1}{2}\right)\left(\|u\|_{p}^{p}\right)^{\gamma-1}< 0.$$
	Hence 
	$$ \|u\|_{p}> \left(\frac{1}{C(\frac{1}{2})}\right)^{\frac{1}{p(\gamma-1)}}.$$
	The third statement can be shown from the first and the second statements. The proof is now complete.
\end{proof}
As a byproduct of the last lemmas is the following corollary.
\begin{corollary}
We have $\mathcal{N}\neq \emptyset $. Moreover, $E$ is coercive on $\mathcal{N}$.	
\end{corollary}
\begin{proof}
From Lemma \ref{le2}, it is clear that $\mathcal{N}$ is not empty. Now we claim that $E$ is coercive on $\mathcal{N}$. Indeed, assuming that $u\in \mathcal{N}$. By $\eqref{u1}$ we have   
\begin{equation}
E(u)=\frac{1}{p^{2}}\|u\|_{p}^{p}
\end{equation}
On the other hand, the definition of $\mathcal{N}$ combined with $\eqref{eg11}$ gives  
$$[u]^{p}_{s,p}+\|u\|_{p}^{p}=\int_{\Omega} |u|^{p}\log(|u|)\,dx\leq \varepsilon [u]_{s,p}^{p}+C(\varepsilon)\left(\|u\|^{p}\right)^{\gamma},\;\forall \varepsilon >0. $$
This implies 
$$ \|u\|_{p}^{p}\geq \left(\frac{1-\varepsilon}{C(\varepsilon)}\right)^{1/\gamma}\left([u]_{s,p}^{p}\right)^{1/\gamma}.$$
Therefore, $E$ is coercive on $\mathcal{N}$.
\end{proof}
In what follows we show that the infimum in $\eqref{V1}$
is attained by some $u\in \mathcal{N}$ which nontrivial critical point of $E$ and therefore the stationary problem associated with $\eqref{eq1}$ admits a ground state solution.
\begin{lemma}
	The following statements hold :
	\begin{description}
		\item[1] $d=\inf_{u\in X_{0}}\sup_{\lambda > 0}E(\lambda u)$.
		\item[2] $d$ has a postive lower bound, namely,
		\begin{equation}
		d \geq M\;\;\text{with}\; M=\frac{R^{p}}{p^{2}}.
		\end{equation}
		\item[3] There exist an extremal of the variation problem $\eqref{V1}$. More precisely, there is a function $u\in \mathcal{N}$ such that $E(u)=d$.
	\end{description}
\end{lemma}
\begin{proof}
	Let $u\in X_{0}$, according to Lemma \ref{le2}, we have 
	\begin{equation}\label{x2}
	\sup_{\lambda >0}E(\lambda u)=E(\lambda^{*}u)=\frac{1}{p}I(\lambda^{*}u)+\frac{1}{p^{2}}\|\lambda ^{*}u\|_{p}^{p}=\frac{1}{p^{2}}\|\lambda ^{*}u\|_{p}^{p}.
	\end{equation}
	From the definition of $\mathcal{N}$ and Lemma \ref{le2}, we have $\lambda^{*}u\in \mathcal{N}$. Thus 
	\begin{equation}\label{x1}
	E(\lambda^{*}u)\geq d=\inf_{u\in \mathcal{N}} E(u).
	\end{equation}
Combining $\eqref{x2}$ and $\eqref{x1}$ we obtain  
	\begin{equation}\label{x4}
	\inf_{u\in X_{0}}\sup_{\lambda > 0}E(\lambda u)\geq d.
	\end{equation}
	On the other hand, if $u\in \mathcal{N}$ then by using $\eqref{d4}$ we get the only critical point in $(0,+\infty) $ of the mapping $j(\lambda)$ is $\lambda^{*}=1$. Thus 
	$$ \sup_{\lambda >0}E(\lambda u)=E(u),$$
	for each $u\in \mathcal{N}$. Hence
	\begin{equation}\label{x3}
	\inf_{u\in X_{0}}\sup_{\lambda > 0}E(\lambda u)\leq \inf_{u\in \mathcal{N}}\sup_{\lambda > 0}E(\lambda u)=\inf_{u\in \mathcal{N}} E(u)=d.
	\end{equation} 
	Thereby, the first statement follows from $\eqref{x4}$ and $\eqref{x3}$.
	From Lemma \ref{le2}, we have $I(\lambda^{*}u)=0$. This implies 
	\begin{equation}\label{x5}
	\|\lambda^{*}u\|_{p}\geq R=\left(\frac{1}{C(\frac{1}{2})}\right)^{\frac{1}{p(\gamma-1)}}
	\end{equation}
by Lemma $\eqref{le3}$. The last inequality combined with $\eqref{x2}$, yields 
	$$ \sup_{\lambda >0} E(\lambda u)\geq \frac{R^{p}}{p^{2}}=M.$$
	Thus, it turns out that $d\geq M$.\\
	In order to show the third statement. Let $\{u_{k}\}^{\infty}_{k=1}\subset \mathcal{N}$ be a minimizing sequence for $E$ such that 
	$$ \lim_{k\rightarrow \infty}E(u_{k})=d.$$
	On the other hand, we have already shown that $E$ is coercive on $\mathcal{N}$. Thus $ \{u_{k}\}^{\infty}_{k=1}$ is bounded in $W_{0}^{s,p}(\Omega)$. Since $W_{0}^{s,p}(\Omega)\hookrightarrow L^{p}(\Omega)$ is compact embedding, there exists a function $u$ and a subsequence of $ \{u_{k}\}^{\infty}_{k=1}$, still denoted by $\{u_{k}\}^{\infty}_{k=1}$, such that 
	\begin{equation}\label{Cn}
	\left\{\begin{array}{lll}
	 u_{k}\rightharpoonup u &\text{in} & W_{0}^{s,p}(\Omega),\\
	 u_{k}\rightarrow u &\text{in}& L^{p}(\Omega),\\
	 u_{k}\rightarrow u &\text{a.e. in }&\Omega.
	\end{array}\right.
	\end{equation}
Now we claim that 
	\begin{equation}\label{Cla}
	\lim_{k\rightarrow \infty}\int_{\Omega} |u_{k}|^{p}\log(|u_{k}|)\,dx= \int_{\Omega} |u|^{p}\log(|u|)\,dx.
	\end{equation}
	Indeed,	from $\eqref{Cn}$ clearly this implies 
		\begin{equation}\label{z77}
		|u_{k}|^{p-2}u_{k}\log(|u_{k}|)\rightarrow |u|^{p-2}u\log(|u|),\;\; a.e\; x\in \Omega.
		\end{equation}
		Taking $\varrho=\frac{sp(p-1)}{N-sp}$ in lemma \ref{le1} and by a straightforward competition we have 		
		\begin{eqnarray}\label{z88}
		\nonumber\int_{\Omega} \left| |u_{k}|^{p-2}u_{k}\log(|u_{k}|)\right|^{\frac{p}{p-1}}\,dx&=&\int_{|u_{k}|\leq 1} \left| |u_{k}|^{p-2}u_{k}\log(|u_{k}|)\right|^{\frac{p}{p-1}}\,dx+\int_{|u_{k}|>1} \left| |u_{k}|^{p-2}u_{k}\log(|u_{k}|)\right|^{\frac{p}{p-1}}\,dx,\\\nonumber
		&\leq& c|\Omega|+C\int_{\Omega_{2}} |u_{k}|^{p_{s}^{*}}\,dx,\\
		&\leq & c|\Omega|+C_{1}[u_{k}]_{s,p}\leq C.
		\end{eqnarray}	
	Here we have used the continuous embedding $W_{0}^{s,p} (\Omega)\hookrightarrow L^{p_{s}^{*}}(\Omega).$ 
	Using \cite[Lemma 1.3, p. 12]{L} we conclude that 
	$$ |u_{k}|^{p-2}u_{k}\log(|u_{k}|)\rightarrow |u|^{p-2}u\log(|u|)\;\text{weakly in}\; L^{\frac{p}{p-1}}(\Omega).$$
	On the other hand, we have 
	\begin{eqnarray*}
		&&\left| \int_{\Omega} |u_{k}|^{p}\log(|u_{k}|)\,dx - \int_{\Omega} |u|^{p}\log(|u|)\,dx\right|\\
		&\leq &\left|\int_{\Omega} (u_{k}-u)|u_{k}|^{p-2}u_{k}\log(|u_{k}|)\,dx\right|+\left|\int_{\Omega} u\left[|u_{k}|^{p-2}u_{k}\log(|u_{k}|)-|u|^{p-2}u\log(|u|)\right]\,dx\right|\\
		&\leq &
		C\|u_{k}-u\|_{p}+\left|\int_{\Omega} u\left(|u_{k}|^{p-2}u_{k}\log(|u_{k}|)-|u|^{p-2}u\log(|u|)\right)\,dx\right|\rightarrow 0 \;\text{as}\; k\rightarrow \infty.
	\end{eqnarray*}
	Therefore, the claim holds. Using the weak lower semicontinuity of the norm on $W_{0}^{s, p}(\Omega)$, we deduce
	\begin{eqnarray*}
		E(u)&=&\frac{1}{p}[u]_{s,p}^{p}+\frac{1}{p}\|u\|^{p}_{p}-\int_{\Omega}|u|^{p}\log(|u|)\,dx+\frac{1}{p^{2}}\|u\|^{p}_{p}\\
		&\leq & \liminf_{k\rightarrow\infty}\left(\frac{1}{p}[u_{k}]_{s,p}^{p}+\frac{1}{p}\|u_{k}\|^{p}_{p}-\int_{\Omega}|u_{k}|^{p}\log(|u_{k}|)\,dx+\frac{1}{p^{2}}\|u_{k}\|^{p}_{p} \right)\\
		&=&\liminf_{k\rightarrow\infty} E(u_{k})=d.
	\end{eqnarray*}
	Thanks to $u_{k}\in \mathcal{N}$ we have $u_{k}\in X_{0}$ and $I(u_{k})=0$, then by Lemma \ref{le3} we get 
	$$ \|u_{k}\|_{p}\geq R.$$
	Hence, by strong convergence in $L^{p}(\Omega)$ it turns out that $\|u\|_{p}\neq 0$, thus $u\in X_{0}$. Furthermore, the weak lower semicontinuity of the norm on $W_{0}^{s, p}(\Omega)$ ensures that 
	\begin{eqnarray*}
		I(u)&=&[u]_{s,p}^{p}+\|u\|^{p}_{p}-\int_{\Omega} |u|^{p}\log(|u|)\,dx \\
		&\leq& \liminf_{k\rightarrow \infty}\left([u_{k}]_{s,p}^{p}+\|u_{k}\|^{p}_{p}-\int_{\Omega} |u_{k}|^{p}\log(|u_{k}|)\,dx \right)\\
		&=& \liminf_{k\rightarrow \infty} I(u_{k})=0
	\end{eqnarray*}
	It remains to show that $I(u)=0$. Arguing by contradiction, if this is not true then we have $I(u)< 0$. By Lemma \ref{le2}, there exists a positive constant $\lambda^{*}$ such that  
	$$ \lambda^{*}=\lambda^{*}(u)=\exp\left(\frac{[u]_{s,p}^{p}+\|u\|^{p}-\int_{\Omega} |u|^{p}\log(|u|)\,dx}{\|u\|^{p}}\right)< 1$$
	and satisfying $I(\lambda^{*}u)=0$. Therefore, by definition of $d$ we obtain 
	$$ 0< d\leq E(\lambda^{*}u)=\frac{1}{p^{2}}\|\lambda^{*}u\|_{p}^{p}\leq \frac{(\lambda^{*})^{p}}{p^{2}}\liminf_{k\rightarrow \infty}\|u_{k}\|_{p}^{p}=(\lambda^{*})^{p}\liminf_{k\rightarrow \infty}E(u_{k})=(\lambda^{*})^{p}d< d,$$
	but this is a contradiction. Thus, the proof is now complete.	
\end{proof}
\section{Proof of theorem \ref{th1}}
In this section, we prove the local existence of weak solutions for problem $\eqref{eq1}$. The proof will be done by using the Galerkin approximation with compactness methods. Therefore, we divide the proof into its naturally arising steps. 
\begin{description}
		\item[Step 1] From the continuous emebding $W_{p}^{s, p}(\Omega)\hookrightarrow L^{p_{s}^{*}}(\Omega)$  we have the Gelfand triple 
		$$ W^{s,p}_{0}(\Omega)\hookrightarrow^{c,d} L^{2}(\Omega)\hookrightarrow^{c,d} (W^{s,p}_{0}(\Omega))^{*}.$$
		
		Here, $\hookrightarrow^{c,d} $ denotes a dense and compact embedding. Let $\{V_{m}\}_{m\in \mathbb{N}}$  be a Galerkin scheme of the separable  Banach space  $W_{0}^{s,p}(\Omega)$, i.e, 
		\begin{equation}
		V_{m}=Span\{\varphi_{1}, \varphi_{2}, \ldots, \varphi_{m}\},\;\; \overline{\bigcup_{m\in \mathbb{N}}V_{m}}^{W_{0}^{s,p}(\Omega)}=W_{0}^{s,p}(\Omega), 
		\end{equation}
		with $\{\varphi_{j}\}_{j=1}^{\infty}$ is an orthonormal basis in $L^{2}(\Omega)$ . Without loss of generality, we assume that for each $m\in \mathbb{N}$ the space $V_{m}$ is a subset of $W^{s+\eta(p-1),p}_{0}(\Omega)$ (note that $C^{\infty}_{0}(\Omega)$ is dense in $W^{s,p}_{0}(\Omega)$ ) for $\eta > 0$ chosen later in the proof. Let $u_{0}\in X_{0}$ then we can find $u_{0m}\in V_{m}$ such that
		\begin{equation}\label{CVV}
		u_{m}(0)=u_{0m}\rightarrow u_{0}\;\;\text{strongly in }\; W_{0}^{s,p}(\Omega)\;\text{as}\; m\rightarrow \infty.
		\end{equation}
	For each $m$, we look for the approximate solutions $ u_{m}(x,t)=\sum_{j=1}^{m}g_{jm}(t)\varphi_{j}(x)$ satisfying the following identities :
		 \begin{equation}\label{g1}
		\int_{\Omega} u_{mt}(t)\varphi_{i}\,dx+K^{s,p}(u_{m}(t), \varphi_{i})+\int_{\Omega} |u_{m}(t)|^{p-2}u_{m}(t)\varphi_{i}\,dx=\int_{\Omega} |u_{m}(t)|^{p-2}u_{m}(t)\log(|u_{m}(t)|)\varphi_{i}\,dx,
		\end{equation}
		with the initial conditions
		\begin{equation}\label{g2}
		u_{m}(0)=u_{0m},
		\end{equation}
		Then $\eqref{g1}-\eqref{g2}$ is equivalent to the following initial value problem for a system of nonlinear ordinary differential equations on $g_{im}$ : 
		\begin{equation}\label{sys}
		\left\{
		\begin{array}{l}
		g'_{im}(t)=F_{i}(g(t)),\;i=1,2, \ldots, m,\;t\in[0,t_{0}],\\
		g_{im}(0)=a_{im},, \;\;i=1,2,\ldots, m,
		\end{array}\right.
		\end{equation}
		where $F_{i}(g(t))=-K^{s,p}(u_{m}(t), \varphi_{i})-\int_{\Omega} |u_{m}(t)|^{p-2}u_{m}(t)\varphi_{i}\,dx+\int_{\Omega} |u_{m}(t)|^{p}\log(|u_{m}(t)|)\varphi_{i}\,dx.$
		By the Picard iteration method, there is $t_{0,m}> 0$ depending on $|a_{im}|$ such that problem $\eqref{sys}$ admits a unique local solution $g_{im}\in C^{1}([0,t_{0,m}])$.
\item[Step 2] Multiplying the $i^{th}$ equation in $\eqref{g1}$ by $g_{im}(t)$ and summing over $i$ from $1$ to $m$, we obtain  
		\begin{equation}\label{g3}
		\frac{1}{2}\frac{d}{dt}\|u_{m}(t)\|_{2}^{2}+[u_{m}(t)]_{s,p}^{p}+\|u_{m}(t)\|^{p}_{p}=\int_{\Omega} |u_{m}(t)|^{p}\log(|u_{m}(t)|)\,dx.
		\end{equation}
		On the other hand, from $\eqref{rrr1}$ we have 
		\begin{eqnarray}\label{r1}
		\nonumber\int_{\Omega} |u_{m}(t)|^{p}\log(|u_{m}(t)|)\,dx&\leq&
		\frac{1}{\varrho} \|u_{m}(t)\|^{p+\varrho}_{p+\varrho}, \;\,\forall t\in [0, t_{0,m}].
		\end{eqnarray}
		where $\varrho$ is chosen sufficiently small such that $0< \varrho< \frac{2sp}{N}$. Since $p<\varrho+p< p^{*}_{s} $ and using Lemma \ref{lemin} 
		with the continuous embedding $ W_{0}^{s,p}(\Omega)\hookrightarrow  L^{p_{s}^{*}}(\Omega),$
		we deduce 
		\begin{eqnarray}\label{I1}
		\nonumber\int_{\Omega} |u_{m}(t)|^{p}\log(|u_{m}(t)|)\,dx &\leq& C\|u_{m}(t)\|_{p_{s}^{*}}^{\theta(p+\varrho)}\|u_{m}(t)\|_{2}^{(1-\theta)(p+\varrho)}\\
		&\leq &C_{1} [u_{m}(t)]_{s,p}^{\theta(p+\varrho)}\|u_{m}(t)\|_{2}^{(1-\theta)(p+\varrho)}, \;\,\forall t\in [0, t_{0,m}].
		\end{eqnarray}
		where $\theta \in (0,1)$ satisfies 
		$$ \frac{1}{p+\varrho}=\frac{\theta}{p^{*}_{s}}+\frac{1-\theta}{2},$$
		The above choice of  $\varrho$ ensures that  $\theta(p+\varrho)< p$. For any $\varepsilon \in (0,1)$, the Young inequality yields   
		\begin{eqnarray}\label{I2}
		\int_{\Omega} |u_{m}(t)|^{p}\log(|u_{m}(t)|)\,dx &\leq& \varepsilon [u_{m}(t)]_{s,p}^{p}+C_{\varepsilon}\left(\|u_{m}(t)\|_{2}^{2}\right)^{\gamma},\;\forall t\in [0, t_{0,m}]
		\end{eqnarray}
		where $\gamma =\frac{p(1-\theta)(p+\varrho)}{2[p-\theta(p+\rho)]}> 1.$ Combining 
		$\eqref{g3}$ with $\eqref{I2}$ we get 
		\begin{equation}\label{ggg3}
		\frac{1}{2}\frac{d}{dt}\|u_{m}(t)\|_{2}^{2}+(1-\varepsilon)[u_{m}(t)]_{s,p}^{p}+\|u_{m}(t)\|^{p}_{p}\leq C_{\varepsilon}\left(\|u_{m}(t)\|_{2}^{2}\right)^{\gamma}.
		\end{equation}
		Taking $\varepsilon=1/2$ in $\eqref{ggg3}$ we conclude that 
		\begin{equation*}
		\frac{d}{dt}\|u_{m}(t)\|_{2}^{2}+[u_{m}(t)]_{s,p}^{p}+\|u_{m}(t)\|^{p}_{p}\leq C_{2}\left(\|u_{m}(t)\|_{2}^{2}\right)^{\gamma}.
		\end{equation*}
		Since $\gamma >1$, the above inequality yields 
		$$\|u_{m}(t)\|^{2}_{2}\leq \left[C_{3}^{(1-\gamma)}-C_{2}(\gamma-1)t\right]^{\frac{1}{1-\gamma}},$$
		only if $t<T_{0}=\frac{C^{(1-\gamma)}_{3}}{C_{2}(\gamma-1)}$, where $C_{3}=\sup\limits_{m\in \mathbb{N}^{*}}\|u_{0m}\|^{2}_{2}$. It follows that 
		\begin{equation*}\label{ety}
		\|u_{m}(t)\|^{2}_{2}\leq 2^{\frac{1}{\gamma-1}}C_{3},\;\;\forall t\leq \min\left\{t_{0,m}, T_{0}/2\right\}.
		\end{equation*}
		Therefore 
		\begin{equation*}\label{ety}
		\|u_{m}(t_{0,m})\|^{2}_{2}\leq 2^{\frac{1}{\gamma-1}}\left(C_{3}+1\right).
		\end{equation*}
		Thus, we can replace $u_{0m}$ in $\eqref{sys}$ by $u_{m}(x,t_{0,m})$ and extend the solution to the interval $[0,T_{0}/2]$ by repeating the above process. We deduce 
		\begin{equation}\label{EQQ}
		\|u_{m}(t)\|_{2}\leq 2^{\frac{1}{\gamma-1}}C_{3}, \;\forall t\in [0, T_{*}],\; \Big(T_{*}=T_{0}/2\Big).
		\end{equation}
		Inserting $\eqref{EQQ}$ into $\eqref{I2}$, we have for any $\varepsilon \in (0, 1)$
		\begin{eqnarray}\label{III2}
		\int_{\Omega} |u_{m}(t)|^{p}\log(|u_{m}(t)|)\,dx &\leq& \varepsilon [u_{m}(t)]_{s,p}^{p}+2^{\frac{2\gamma}{\gamma-1}}C_{\varepsilon}C^{2\gamma}_{3},\;\forall t\in [0, T_{*}].
		\end{eqnarray}
	 We next multiply both sides of $\eqref{g1}$ by $g'_{im}(t)$, take the sum over $i\in \{1,2,\ldots,m\}$, and afterwards integrate over $(0,t)$ yields
		\begin{equation}\label{y1}
		\int_{0}^{t}\|u_{ms}(s)\|_{2}^{2}\,ds+E(u_{m}(t))=E(u_{m}(0))
		\end{equation}
		Notice that by $\eqref{CVV}$ and the continuity of $E$ there exists a postive constant $C$ such that 
		\begin{equation}\label{y2}
		E(u_{m}(0))\leq C,\;\;\text{for all}\; m. 
		\end{equation}
	Combining the definition of $E$ with $\eqref{III2}$ we obtain 
		\begin{equation}\label{y3}
		E(u_{m}(t))\geq \frac{1-\varepsilon}{p}[u_{m}(t)]_{s,p}^{p}+\frac{1}{p}\|u_{m}(t)\|_{p}^{p}+\frac{1}{p^{2}}\|u_{m}(t)\|_{p}^{p}-\frac{2^{\frac{2\gamma}{\gamma-1}}C_{\varepsilon}C^{2\gamma}_{3}}{p}.
		\end{equation}
		From $\eqref{y1}-\eqref{y3}$, it follows that 
		\begin{equation}\label{y4}
		\|u_{m}\|_{L^{\infty}(0,T_{*},W_{0}^{s,p}(\Omega))}\leq C, 
		\end{equation}
		and 
		\begin{equation}\label{y5}
		\|u_{mt}\|_{L^{2}(0,T_{*},L^{2}(\Omega)}\leq C.
		\end{equation}
		\item[Step 3]
		Combining a priori estimates $\eqref{y4}$ and $\eqref{y5}$ we get the existence of a function $u$ and a subsequence of $\{u_{m}\}^{\infty}_{m=1}$ still denoted by $\{u_{m}\}^{\infty}_{m=1}$ such that 
		\begin{equation}\label{z1}
		u_{m}\rightarrow u \;\;\text{weakly}^{*}\; \text{in} \;\;L^{\infty}(0,T_{*},W_{0}^{s,p}(\Omega)),
		\end{equation}
		\begin{equation}\label{z2}
		u_{mt}\rightarrow u_{t} \;\;\text{weakly in} \;\;L^{2}(0,T_{*},L^{2}(\Omega)),
		\end{equation}
		\begin{equation}\label{z3}
		(-\Delta)^{s}_{p}u_{m}\rightarrow \chi \;\;\text{weakly}^{*}\;\text{ in} \;\;L^{\infty}(0,T_{*},W^{-s,p'}(\Omega))
		\end{equation}
		Since $\{u_{m}\}^{\infty}_{m=1}\subset L^{\infty}(0,T_{*},W_{0}^{s,p}(\Omega))$ and $\{u_{mt}\}^{\infty}_{m=1 }\subset L^{2}(0,T_{*},L^{2}(\Omega))$, Aubin-Lions compacteness theorem  \cite[Theorem 5.1, p. 58]{L} implies that up to a subsequence,
		\begin{equation}\label{z4}
		u_{m}\rightarrow u \;\text{stronly in }\; C([0,T_{*}], L^{r}(\Omega)),\;\;\forall r\in [2,p_{s}^{*}).
		\end{equation}
		Let $0< \eta < \min\left(s, \frac{1-s}{p-1}\right)$ be chosen as in Lemma \ref{lem1}, then by using Aubin-Lions compacteness theorem again we have 
		$$ L^{p}  (0,T_{*}, W_{0}^{s,p}(\Omega))\cap W^{1,2}(0,T_{*},L^{2}(\Omega))\hookrightarrow^{c} L^{p}(0,T_{*}, W_{0}^{s-\eta,p}(\Omega)).$$
		Therefore, we obtain 
		\begin{equation}\label{z5}
		u_{m}\rightarrow u \;\text{stronly in }\; L^{p}(0,T_{*}, W_{0}^{s-\eta,p}(\Omega)),
		\end{equation}
		\begin{equation}\label{z6}
		u_{m}(t)\rightarrow u(t) \;\text{stronly in }\;  W_{0}^{s-\eta,p}(\Omega),\;\; a.e \;\text{in}\; (0,T_{*}).
		\end{equation}
		By $\eqref{z4}$, we have 
		\begin{equation}\label{z7}
		|u_{m}|^{p-2}u_{m}\log(|u_{m}|)\rightarrow |u|^{p-2}u\log(|u|),\;\; a.e\; (x,t)\in \Omega \times (0,T_{*}).
		\end{equation}
		A straightforward computation yields 
		\begin{eqnarray}\label{z8}
		\nonumber\int_{\Omega} \left| |u_{m}|^{p-2}u\log(|u_{m}|)\right|^{p'}\,dx&=&\int_{\Omega_{1}} \left| |u_{m}|^{p-2}u\log(|u_{m}|)\right|^{p'}\,dx+\int_{\Omega_{2}} \left| |u_{m}|^{p-2}u\log(|u_{m}|)\right|^{p'}\,dx\\\nonumber
		&\leq& e^{-p'}|\Omega|+C\int_{\Omega_{2}} |u_{m}(t)|^{q}\,dx\\
		&\leq & e^{-p'}|\Omega|+C_{1}[u_{m}(t)]^{p}_{s,p}\leq C_{T_{*}}
		\end{eqnarray}
		where $q\in [p,p^{*}_{s}]$, $p'=\frac{p}{p-1}$ and  $\Omega_{1}:=\{x\in \Omega,\; |u_{m}(x,t)|\leq 1\}$, $\Omega_{2}:=\{x\in \Omega,\; |u_{m}(x,t)|> 1\}$. Hence, by using [\cite{L}, Lemma 1.3, p. 12], it follows from $\eqref{z7}$  and $\eqref{z8}$ that 
		\begin{equation}\label{LOG}
		|u_{m}|^{p-2}u_{m}\log (|u_{m}|)\rightarrow |u|^{p-2}u\log (|u|), \;\; \text{weakly}^{*}\;\text{in}\; L^{\infty}(0,T_{*},L^{p'}(\Omega)).
		\end{equation}
		In the same way, one can show that  
			\begin{equation}\label{LOG2}
		|u_{m}|^{p-2}u_{m}\rightarrow |u|^{p-2}u, \;\; \text{weakly}^{*}\;\text{in}\; L^{\infty}(0,T_{*},L^{p'}(\Omega)).
		\end{equation}
	From $\eqref{z1}-\eqref{z2}$ and \cite[see, Lemma 3.1.7]{Zh}, 
		$$ u_{m}(0)\rightarrow u(0)\;\text{weakly}\;\text{in}\; L^{2}(\Omega). $$
		However, by (\ref{CVV}) we know that $u_m(0) \to u_0$ in $W_{0}^{s, p}(\Omega)$, in particular $u_m(0) \to u_0$ in $L^{2}(\Omega)$, and so, $u(0)=u_{0}$. This shows that $u$ satisfies the initial condition.\\
\end{description}
	It remains to pass to the limit in $\eqref{g1}$. For $\phi \in L^{2}(0,T_{*})$ and $\varphi_{i}\in V_{m}$, we have 
	\begin{eqnarray}\label{LIM}
	\nonumber\int_{0}^{T_{*}}\int_{\Omega}u_{mt}(t)\varphi_{i}\,dx\phi(t)\,dt+\int_{0}^{T_{*}}K^{s,p}(u_{m}(t),\varphi_{i})\phi(t)\,dt+\int_{0}^{T_{*}}\int_{\Omega} |u_{m}(t)|^{p-2}u_{m}(t)\varphi_{i}\,dx\phi(t)\,dt&=&\\
	=\int_{0}^{T_{*}}\int_{\Omega}|u_{m}(t)|^{p-2}u_{m}(t)\log(|u_{m}(t)|)\varphi_{i}\,dx\phi(t)\,dt.
	\end{eqnarray} 
Since $V_{m}\subset W_{0}^{s+\eta(p-1),p}(\Omega)$ (from the first step) then by using lemma \ref{lem1} together with  $\eqref{z6}$, we get 
	$$ K^{s,p}(u_{m}(t), \varphi_{i}) \rightarrow K^{s,p}(u(t),\varphi_{i})\;\;\text{a.e. in}\; (0,T_{*})$$
	Moreover, proposition \ref{prp1} and the boundedness of $\{u_{m}\}^{\infty}_{m=1}$ in $L^{\infty}(0,T_{*},W_{0}^{s,p}(\Omega))$ ensure that the sequence $\{K^{s,p}(u_{m}(t), \varphi_{i})\}^{\infty}_{m=1}$ is bounded in $L^{\infty}(0,T_{*})$. Therefore, the Lebesgue dominated convergence theorem implies 
	\begin{equation}\label{LOG3}
	\int_{0}^{T_{*}}K^{s,p}(u_{m}(t),\varphi_{i})\phi(t)\,dt\rightarrow \int_{0}^{T_{*}}K^{s,p}(u(t),\varphi_{i})\phi(t)\,dt\;\text{as}\; m\rightarrow\infty.
	\end{equation}
	Letting $m\rightarrow\infty$ in $\eqref{LIM}$ and using $\eqref{z2}$, $\eqref{LOG}$, $\eqref{LOG2}$, $\eqref{LOG3}$ we deduce 
	$$ \int_{0}^{T_{*}}\int_{\Omega}u_{t}(t)\varphi_{i}\,dx\phi(t)\,dt+\int_{0}^{T_{*}}K^{s,p}(u(t),\varphi_{i})\phi(t)\,dt+\int_{0}^{T_{*}}\int_{\Omega} |u(t)|^{p-2}u(t)\varphi_{i}\,dx\phi(t)\,dt=$$
	$$=\int_{0}^{T_{*}}\int_{\Omega}|u(t)|^{p-2}u(t)\log(|u(t)|)\varphi_{i}\,dx\phi(t)\,dt$$
	Thus 
	$$ \int_{\Omega}u_{t}(t)\varphi_{i}\,dx+K^{s,p}(u(t),\varphi_{i})+\int_{\Omega} |u(t)|^{p-2}u(t)\varphi_{i}\,dx=
	\int_{\Omega}|u(t)|^{p-2}u(t)\log(|u(t)|)\varphi_{i}\,dx,\;\text{a.e. in}\; (0,T_{*}),$$
	Using the density of $V_{m}$ in $W_{0}^{s,p}(\Omega)$ we obtain
	$$ \int_{\Omega}u_{t}(t)v\,dx+K^{s,p}(u(t),v)+\int_{\Omega} |u(t)|^{p-2}u(t)v\,dx=
	\int_{\Omega}|u(t)|^{p-2}u(t)\log(|u(t)|)v\,dx,\;\text{a.e. in}\; (0,T_{*}),\;\forall v\in W_{0}^{s,p}(\Omega).$$
	We now show that the solution $u$ satisfies the energy inequality $\eqref{en1}$. To do this end, let $\theta$ be the nonnegative function which belongs to $C([0,T_{*}])$. From $\eqref{y1}$ we have  
	\begin{equation}\label{en2}
	\int_{0}^{T_{*}}\theta(t)\,dt\int_{0}^{T_{*}}\|u_{ms}(s)\|_{2}^{2}\,ds+\int_{0}^{T_{*}}E(u_{m}(t))\theta(t)\,dt=\int_{0}^{T_{*}}E(u_{m}(0))\theta(t)\,dt 
	\end{equation}
	The right-hand side of $\eqref{en2}$ converges to 
	$$ \int_{0}^{T_{*}}E(u_{0})\theta(t)\,dt$$
	as $m\rightarrow \infty$. The second term in the left-hand side $\int_{0}^{T}E(u_{m}(t))\theta(t)\,dt$ is lower semicontinuous with respect to the weak topology of $W_{0}^{s,p}(\Omega)$. Hence 
	
	\begin{equation}\label{est}
	\int_{0}^{T_{*}}E(u(t))\theta(t)\,dt\leq \liminf_{m\rightarrow \infty} \int_{0}^{T_{*}}E(u_{m}(t))\theta(t)\,dt 
	\end{equation}
	Therefore, we obtain  
	$$ \int_{0}^{T_{*}}\theta(t)\,dt\int_{0}^{t}\|u_{s}(s)\|^{2}_{2}\,ds+ \int_{0}^{T_{*}}E(u(t))\theta(t)\,dt\leq  \int_{0}^{T_{*}}E(u_{0})\theta(t)\,dt$$
	Since $\theta$ was arbitrarily chosen we conclude 
	$$ \int_{0}^{t}\|u_{s}(s)\|^{2}_{2}\,ds+ E(u(t))\leq  E(u_{0}),\;\;\;a.e.\, t\in [0,T_{*}].$$
Thus, this completes the proof of theorem \ref{th1}.
\section{Proof of Theorem \ref{th3}}
In this section, by using the potential well theory combined with the \textit{Nehari} manifold, we prove that the local weak solutions of problem $\eqref{eq1}$ exist globally, see (\cite{CT}, \cite{FU}) and the references therein for some results on global existence of solutions. Furthermore, we show that the norm $\|u(t)\|_{2}$ decays polynomially. For this purpose, we need to recall the following Lemma due to Martinez \cite{M}.
\begin{lemma}\label{le4}
	Let $f :\mathbb{R}^{+}\rightarrow\mathbb{R}^{+}$ be a nonincreasing function and $\sigma $ is a nonnegative constant such that 
	$$ \int_{t}^{+\infty}f^{1+\sigma}(s)\,ds\leq \frac{1}{\omega}f^{\sigma}(0)f(t), \;\; \forall t\geq 0.$$
	then we have 
	\begin{description}
		\item[1] $f(t)\leq f(0)e^{1-\omega t}$, for all $t\geq 0$, whenever $\sigma =0$.
		\item [2] $f(t)\leq f(0)\left(\frac{1+\sigma}{1+\omega \sigma t}\right)^{1/\sigma}$, for all $t\geq0$, whenever $\sigma > 0$.
	\end{description}
\end{lemma}
Notice that, by the assumption that 
$ u_{0}\in \mathcal{W}$ we obtain  
$$
E(u_{0}) =\frac{1}{p}I(u_{0})+\frac{1}{p^{2}}\|u_{0}\|^{p}_{p} >0.
$$
From $\eqref{y1}$ we have 
\begin{equation}\label{EZQ}
	\int_{0}^{t}\|u_{ms}(s)\|_{2}^{2}\,dx+E(u_{m}(t))=E(u_{m}(0)), \; 0\leq t< T_{m}, 
	\end{equation}
	where $T_{m}$ is the maximal time of existence of solution $u_{m}(x,t)$. Since  $u_{0m}$ converges to $u_{0}$ strongly in $W_{0}^{s, p}(\Omega)$, the continuity of $E$ ensures that   
	$$ E(u_{m}(0))\rightarrow E(u_{0}),\;\text{as}\; m\rightarrow +\infty. $$
	From the assumption that $E(u_{0})< d$,  we have $E(u_{0m})< d$, for sufficiently large $m$. This combined with $\eqref{EZQ}$ implies that 
	\begin{equation}\label{ETT}
	\int_{0}^{t}\|u_{ms}(s)\|_{2}^{2}\,dx+E(u_{m}(t))<d, \; 0\leq t< T_{m}, 
	\end{equation}
for sufficiently latge $m$. We will show that $T_{m}=+\infty$ and 
\begin{equation}\label{v12}
u_{m}(t)\in \mathcal{W},\;\;\forall t\geq 0, 
\end{equation}
for sufficiently large $m.$ Suppose by contradiction that $u_{m}(t_1) \notin \mathcal{W}$ for some $t_1 \in[0,T_{m})$. Let $t_{*} \in [0,T_{m})$ be the smallest time for which $u_{m}(t_{*})\notin \mathcal{W}$. Then, by continuity of $u_{m}(t)$, we get $u_{m}(t_{*})\in \partial \mathcal{W}$. Hence, it turns out that
	\begin{equation}\label{v4}
	E(u_{m}(t_{*}))=d.
	\end{equation} 
	or 
	\begin{equation}\label{v2}
	I(u_{m}(t_{*}))=0.
	\end{equation} 
	It is clear that $\eqref{v4}$ could not occur by $\eqref{ETT}$ while if $\eqref{v2}$ holds then, by the definition of $d$, we have 
	$$ E(u_{m}(t_{*}))\geq \inf_{u\in \mathcal{N}}E(u)=d, $$
	which also a contradiction with $\eqref{ETT}$. Consequently, $\eqref{v12}$ is hold.\\
	On other hand, since $u_{m} (t)\in \mathcal{W}$ and 
	$$ E(u_{m}(t))=\frac{1}{p}I(u_{m}(t))+\frac{1}{p^{2}}\|u_{m}(t)\|_{p}^{p},\;\;\forall t\in [0, T_{m}), $$
	we deduce that 
	\begin{equation}\label{i2}
	\|u_{m}(t)\|_{p}^{p}<  dp^{2}\;\;\text{and}\;\;\; \int_{0}^{t}\|u_{ms}(s)\|^{2}_{2}\,ds< d,
	\end{equation}
	for sufficiently large $m$ and $t\in [0,T_{m}).$ Further, by using $\eqref{eg11}$ and for any $\varepsilon \in (0,1)$ we have 
	\begin{eqnarray*}
		[u_{m}(t)]_{s,p}^{p}&=&p  E(u_{m}(t))-\|u_{m}(t)\|_{p}^{p}+\int_{\Omega} |u_{m}(t)|^{p}\log(|u_{m}(t)|)\,dx-\frac{1}{p}\|u_{m}(t)\|_{p}^{p}\\
		&\leq & p  E(u_{m}(t))+\int_{\Omega} |u_{m}(t)|^{p}\log(|u_{m}(t)|)\,dx\\
		&\leq & pd+\varepsilon [u_{m}(t)]^{p}_{s,p}+C(\varepsilon)(dp^{2})^{\gamma}
	\end{eqnarray*}
Hence
	\begin{equation}\label{i1}
	[u_{m}(t)]_{s,p}^{p}\leq C_{d},\;  \forall t\in [0,T_{m}).
	\end{equation}
	The above estimates allow us to take $T_{m}=+\infty$ for all $m$.
Now using $\eqref{i2}$ and $\eqref{i1}$, the existence of global solutions follow as in the first section.

In order to prove that the norm $\|u(t)\|_{2}$ decays polynomially, we need to assume the following condition  
$$E(u_{0})< M,$$
where $M$ was introduced in Lemma \ref{le3}. Combining $\eqref{u1}$ with the fact that $u_{m}(t)\in \mathcal{W}$ we deduce 
	\begin{equation}
	\|u_{m}(s)\|_{p}^{p}\leq p^{2}E(u_{m}(t))\leq p^{2}E(u_{0m}).
	\end{equation}
By $\eqref{eg11}$ for any $\varepsilon \in (0,1),$
	\begin{eqnarray*}
		I(u_{m}(t))&=& [u_{m}(t)]^{p}_{s,p}+\|u_{m}(t)\|_{p}^{p}-\int_{\Omega} |u_{m}(t)|^{p}\log(|u_{m}(t))\,dx\\
		&\geq & \left( 1-\varepsilon\right)[u_{m}(t)]^{p}_{s,p}+\|u_{m}(t)\|_{p}^{p}\left( 1-C(\varepsilon)(\|u_{m}(t)\|_{p}^{p})^{\gamma-1}\right)
	\end{eqnarray*} 
Taking $\varepsilon =\frac{1}{2}$, we get 
	\begin{eqnarray*}\label{o1}
	I(u_{m}(t))\geq \left(1-C\left(\frac{1}{2}\right)\left(p^{2}E(u_{0m})\right)^{\gamma-1}\right)\|u_{m}(t)\|_{p}^{p}&\geq& |\Omega|^{2-p}\left(1-C\left(\frac{1}{2}\right)\left(p^{2}E(u_{0m})\right)^{\gamma-1}\right)\|u_{m}(t)\|_{2}^{p}\\
	&=&\kappa_{m} \|u_{m}(t)\|_{2}^{p}
	\end{eqnarray*}
	
	where $\kappa_{m} =|\Omega|^{2-p}\left(1-C\left(\frac{1}{2}\right)\left(p^{2}E(u_{0m})\right)^{\gamma-1}\right)> 0$. On other hand, multiplying both sides of $\eqref{g1}$ by $g_{im}(t)$, take the sum over $i\in \{1,2,\ldots,m\}$, and afterwards integrate over $(t,T)$ yields 
	\begin{equation*}\label{o2}
	\int_{t}^{T}I(u_{m}(s)\,ds=-\int_{t}^{T}\int_{\Omega} u_{sm}(s)u_{m}(s)\,dxds\leq \frac{1}{2}\|u_{m}(t)\|_{2}^{2},\; 	\forall t\in [0,T].
	\end{equation*}
The last inequality combined with $\eqref{o1}$ yields 
\begin{equation}\label{EST}
	 \int_{t}^{T}\|u_{m}(t)\|_{2}^{p}\,ds\leq \frac{1}{2\kappa_{m}}\|u_{m}(t)\|_{2}^{2}, \;\;\forall t\in [0,T].
\end{equation}
Since $E(u_{0m})\rightarrow E(u_{0})$ as $m\rightarrow\infty$, it follows that $\kappa_{m}\rightarrow \kappa $ as $m\rightarrow\infty$. On the other hand, from $\eqref{i2}$ and $\eqref{i1}$ for each $T>0$ we have $\{u_{m}\}^{\infty}_{m=1}\subset L^{\infty}(0,T, W_{0}^{s, p}(\Omega))$ and $\{u_{mt}\}^{\infty}_{m=1}\subset L^{2}(0,T,L^{2}(\Omega))$. Using Aubin-Lions theorem we conclude 
	$$ \|u_{m}(t)\|_{2}\rightarrow \|u(t)\|_{2}\;\;\text{as}\; m\rightarrow \infty,\; \forall t\in[0,T].$$
Letting $m\rightarrow \infty$ in $\eqref{EST}$ we get 
	\begin{equation}\label{ESTT}
	\int_{t}^{T}\|u(t)\|_{2}^{p}\,ds\leq \frac{1}{2\kappa}\|u(t)\|_{2}^{2}, \;\;\forall t\in [0,T].
	\end{equation}
Now, letting $T\rightarrow +\infty$ and using Lemma \ref{le4}, we obtain
	$$ \|u(t)\|_{2}\leq \|u_{0}\|_{2}\left( \frac{p}{2(1+\kappa (p-2)\|u_{0}\|^{p-2}_{2}t)}\right)^{1/(p-2)}, \;\; t\geq0.$$
	The proof is now complete.
\section{proof of theorem \ref{th4}}
In this section, by means of a differential inequality technique,  we prove that the local solutions of problem $\eqref{eq1}$ blow-up in finite time. 

First, we observe that by \cite[Theorem 2.5.5, p.54]{Zh} the local weak solution $u$ which obtained in section $1$ can be extended to a maximal weak solution in $[0, T_{\max})$. Thus, the energy inquality $\eqref{ENR}$ can be obtained by extending $\eqref{en1}$ to $[0, T_{\max})$. 
Now, we claim that  
	\begin{equation}\label{EZA}
	\text{if}\; u_{0}\in Z\;\;\text{then }\;\; u(t)\in Z,\;\;\forall t\in [0, T_{\max}).
	\end{equation}
	Indeed, arguing as in the proof of theorem \ref{th3}, we get  
	\begin{equation}\label{Sv1}
	u_{m}(t)\in Z,\;\;\forall t\in [0,T_{\max}), 
	\end{equation}
	for sufficiently large $m.$ Using $\eqref{z1}$, $\eqref{z2}$, $\eqref{LOG}$ and the same argument used to obtain $\eqref{en1}$ we conclude 
	$$u(t)\in Z,\;\;\forall t\in [0, T_{\max}).$$
Now we consider the following functional
	\begin{equation}\label{rtt}
	E(t)=\int_{0}^{t}\|u(s)\|_{2}^{2}\,ds+(T-t)\|u_{0}\|^{2}_{2},\;\; t\in [0,T_{\max}).
	\end{equation} 
	By differentiating $E(t) $, we obtain 
	\begin{equation}
	E'(t)=\|u(t)\|_{2}^{2}-\|u_{0}\|^{2}_{2}=\int_{0}^{t}\frac{d}{ds}\left(\|u(s)\|_{2}^{2}\right)\,ds=2\int_{0}^{t}\int_{\Omega} u_{s}(s)u(s)\,dxds
	\end{equation} 
	and 
	\begin{equation}\label{ert1}
	E''(t)=2\int_{\Omega} u_{t}(t)u(t)\,dxdt=-[u(t)]_{s,p}^{p}-\|u(t)\|_{p}^{p}+\int_{\Omega} |u(t)|^{p}\log(|u(t)|)\,dx=-I(u(t))
	\end{equation} 
	It follows from $\eqref{EZA}$ that 
	$$ E''(t) > 0,\;\text{for all}\; t\in [0, T_{\max}).$$ 
	Since $I(u(t))=pE(u(t))-\frac{1}{p}\|u(t)\|^{p}_{p}$ and by using $\eqref{ENR}$ we have 
	\begin{equation}
	E''(t)=-pE(u(t))+\frac{1}{p^{2}}\|u(t)\|_{p}^{p}\geq p\int_{0}^{t}\|u_{s}(s)\|_{2}^{2}\,ds+\frac{1}{p^{2}}\|u(t)\|_{p}^{p}-pE(u_{0})
	\end{equation}
	This implies 
	\begin{equation}
	E''(t)\geq \frac{1}{p^{2}}\|u(t)\|_{p}^{p}\geq |\Omega|^{2-p}\frac{1}{p}\|u(t)\|_{2}^{p}=|\Omega|^{2-p}\frac{1}{p}\left(E'(t)+\|u_{0}\|_{2}^{2}\right)^{p/2}.
	\end{equation}
	Set $\phi(t)=E'(t)+\|u_{0}\|_{2}^{2}$, we conclude that 
	$$ \phi'(t)\geq \frac{|\Omega|^{2-p}}{p} (\phi(t))^{p/2}.$$
	From standard differential inequality, it turns out that 
	$$ \|u(t)\|_{2}^{2}\geq \left( \frac{1}{\|u_{0}\|_{2}^{2-p}-Ct}\right)^{2/(p-2)}$$
	where $C=p^{-1}|\Omega|^{(p-2)}(p-2)$. Therefore 
	$$ \lim_{t\rightarrow T^{-}_{\max}}\|u(t)\|^{2}_{2}=+\infty,\;\; \text{wher}\; T_{\max}=\frac{\|u_{0}\|_{2}^{2-p}}{C}.$$
The proof is now complete.\vspace{0.5cm}\\

\section*{Acknowledgment.} The author warmly thank the anonymous referee for his/her useful and nice comments on the paper.

\end{document}